\documentclass[12pt]{article}%
\usepackage{amsfonts}
\usepackage{amsmath}
\usepackage{amssymb}
\usepackage{graphicx,tikz}%
\usepackage{graphicx}
\providecommand{\U}[1]{\protect\rule{.1in}{.1in}}
\newtheorem{theorem}{Theorem}

\newtheorem{conjecture}[theorem]{Conjecture}
\newtheorem{corollary}[theorem]{Corollary}

\newtheorem{lemma}[theorem]{Lemma}

\newtheorem{proposition}[theorem]{Proposition}

\newenvironment{proof}[1][Proof]{\noindent\textbf{#1.} }{\ \rule{0.5em}{0.5em}}

\DeclareMathOperator{\res}{res}
\DeclareMathOperator{\Span}{span}

\begin{document}

\title{Stability theorems for multiplicities in graded $S_n$-modules}
\author{Marino Romero \thanks{The author was partially supported by the University of
California President's Postdoctoral Fellowship. }
\and Nolan Wallach}
\maketitle

\begin{abstract}
In this paper, we prove several stability theorems for multiplicities of naturally
defined representations of symmetric groups. The first such theorem states
that if we consider the diagonal action of the symmetric group
$S_{m+r}$ on $k$ sets of $m+r$ variables, then the dimension of the invariants of
degree $m$ is the same as the dimension of the invariants of degree $m$ for
$S_{m}$ acting on $k$ sets of $m$ variables. 
Building on this stability, the last section looks at the Hilbert series of coinvariants of the polynomial ring in $k$ sets of $m$ variables. 
We address a conjecture that the Hilbert series, in degrees no more than $m$, can be computed by a truncated power series expression. Using some auxiliary results and manipulations of power series, we show that if this holds for $k$ and $m$, then the truncation gives the correct Hilbert series up to degree $m$ for $k$ sets of $n \geq m$ variables. This shows the validity of the conjecture up to certain degrees.
 We also provide a new equivalent conjecture regarding Gr\"{o}bner bases.

The second type of stability
result is for Weyl modules. We prove that the dimension of the $S_{m+r}$
invariants for a Weyl module ${}_{m+r}F^{\lambda}$ (the Schur-Weyl dual of the
$S_{|\lambda|}$ module $V^{\lambda}$) with $\left\vert \lambda \right\vert  \leq
m$ is of the same dimension as the space of $S_{m}$ invariants for
${}_{m}F^{\lambda}$. Multigraded versions of the first type of result are given, as
are multigraded generalizations to non-trivial modules of symmetric groups.
\end{abstract}

\section{Introduction}

Consider the diagonal action of the symmetric group on $n$ letters, $S_{n}$,
on $k$ sets of variables. A classical theorem of Hermann Weyl asserts that the
polynomial invariants in $k$ sets of variables are generated by the
polarizations of the invariants in one set of variables. Thus, taking a
standard set of generators of the symmetric functions in
$n$ variables, such as the power sums $p_1,\dots, p_n$, and repeatedly polarizing to $k$ sets of $n$ variables, one
gets a set of generators for the invariants for the diagonal action of the
symmetric group on $k$ sets of $n$ variables. This leads to a natural question: what are the
relations among these polynomials (the so-called second fundamental theorem)? This answer is given by the work of Dalbec \cite{Dalbec} or Vaccarino \cite{Vaccarino}, in which the kernel of the projection from $k$ sets of $n+1$ variables to $k$ sets of $n$ variables is given. In particular, their results imply that
 the first relation is in degree $n+2$. We will give an alternate proof of this fact, for it has
several surprising implications. For example, looking upon the  
diagonal action of $S_{n}$ on $k$ sets of variables as the action of $S_{n}$
on $\mathbb{C}^{k}\otimes\mathbb{C}^{n}$ given by sending $s \in S_n $ to $I\otimes s$, 
if $m\leq n$ then
$$
\dim S^{m}(\mathbb{C}^{k}\otimes\mathbb{C}^{m})^{S_{m}}=\dim S^{m}%
(\mathbb{C}^{k}\otimes\mathbb{C}^{n})^{S_{n}}.
$$
Here, if $M$ is an $S_{n}$ module, $M^{S_{n}}$ is the space of invariants.

In this paper all of the results are true over arbitrary fields of
characteristic $0.$ We state them over $\mathbb{C}$, but they can equally well
be stated and proved in the same way over $\mathbb{Q}$.

This stability theorem is equivalent to another of a different sort. For this
we need some notation. If $\lambda,\mu\in\mathbb{Z}^{n}$ then we will say that
$\lambda\succ\mu$ if
$$
\lambda_{1}+...+\lambda_{i}\geq\mu_{1}+...+\mu_{i}, \text{ for } i=1,...,n.
$$
For $\mu\in\mathbb{Z}^{n}$ we define a character of $\left(  \mathbb{C}
^{\times}\right)  ^{n}$, thought of as the subgroup of diagonal matrices in
$GL(n,\mathbb{C})$, by
$$
z\mapsto z^{\mu}=z_{1}^{\mu_{1}}z_{2}^{\mu_{2}}\cdots z_{n}^{\mu_{n}}.
$$
If $\lambda=(\lambda_{1},...,\lambda_{r}) \in\mathbb{Z}^{r}$ with $\lambda
_{1}\geq\lambda_{2}\geq...\geq\lambda_{r}\geq0 $ (that is, $\lambda$ is
dominant) and $r\leq n$, then the theorem of the highest weight implies that
there is a unique up to equivalence irreducible representation of
$GL(n,\mathbb{C})$, ${}_{n}F^{\lambda}$, with the property that if for $\mu\in
\mathbb{Z}^{n}$ there is a nonzero $v\in$ ${}_{n}F^{\lambda}$ such that
$zv=z^{\mu}v$ for $z \in(\mathbb{C}^{*})^{n}$, then $(\lambda_1,\dots,\lambda_r,0,\dots,0) \succ
\mu$.  Another characterization using the Weyl character formula is that
the character of ${}_{n}F^{\lambda}$ is $s_{\lambda}(z_{1},...,z_{n})$, the
Schur symmetric function.

The so-called $GL(k,\mathbb{C})-GL(n,\mathbb{C})$ duality implies that as a
representation of $GL(k,\mathbb{C})\times GL(n,\mathbb{C})$ the space
$S^{r}(\mathbb{C}^{k}\otimes\mathbb{C}^{n})$ is equivalent to%
$$
\bigoplus_{\left\vert \lambda\right\vert =r}{}_{k}F^{\lambda}\otimes{}%
_{n}F^{\lambda}%
$$
where the sum is over dominant $\lambda$ with at most $\min(k,n)$ nonzero entries and
$\left\vert \lambda\right\vert =\sum\lambda_{i} = r$. Using this formalism,
we can give our second main result. We look upon $S_{n}$ as the group of
permutation matrices in $GL(n,\mathbb{C})$, and for $m \leq n$ we identify
$GL(m,\mathbb{C})$ with the group of matrices
$$
\left\{  \left[
\begin{array}
[c]{cc}%
g & 0\\
0 & I
\end{array}
\right]  \Big| ~g\in GL(m,\mathbb{C})\right\}  .
$$

Then for $\left\vert \lambda\right\vert \leq m\leq n$, the following stability
holds:
$$
\dim\left(  {}_{m}F^{\lambda}\right)  ^{S_{m}}=\dim\left(  {}_{n}F^{\lambda
}\right)  ^{S_{n}}.
$$

In the standard Young correspondence between irreducible representations of
$S_{n}$ and partitions of $n$ in the guise of Young diagrams, the trivial
representation of $S_{n}$ corresponds to the partition $(n)$. If $\mu$ is a
partition of $n$, let $V^{\mu}$ denote the corresponding Young representation
of $S_{n}$. Then our result says that
$$
\dim\mathrm{Hom}_{S_{m}}(V^{(m)},{}_{m}F^{\lambda})=\dim\mathrm{Hom}_{S_{m+r}%
}(V^{(m+r)},{}_{m+r}F^{\lambda}).
$$
We show that this
stability can be extended to non-trivial representations of the symmetric
group as follows: 
\\

If $\mu=(\mu_{1},\mu_{2},...)$ is a partition of $m$ and if
$\lambda$ is dominant with
$$
\left\vert \lambda\right\vert \leq m-\mu_{2},
$$
then
$$
\dim\mathrm{Hom}_{S_{m}}(V^{\mu},{}_{m}F^{\lambda})=\dim\mathrm{Hom}_{S_{m+r}%
}(V^{\mu+re_{1}},{}_{m+r}F^{\lambda}),
$$
with $e_{1},...,e_{l}$ the usual basis of $\mathbb{Z}^{l}$. 
(Here if $\mu=(\mu_{1})$, then we take $\mu_{2}=0$.) 
\\

The last section of the paper is devoted to several conjectures that arise
naturally from the results of the earlier sections. The main conjecture
describes the Hilbert series $h_{k,n}(q)$ of the $S_{n}$ coinvariants in $k$ copies of $n$
indeterminates up to degree $n$, and its computational equivalent may also be found in Bergeron's work \cite{Bergeron} (with further details given in the last section).
By coinvariants, we mean the quotient of the polynomial ring by the ideal generated by the homogeneous invariants of positive degree.
 If
\[
\varphi(q)=\sum_{i=0}^{\infty}a_{i}q^{i}%
\]
is a formal power series then we set%
\[
\varphi(q)_{\leq m}=\sum_{i=0}^{m}a_{i}q^{i}.
\]
The conjecture states
\[
h_{k,n}(q)_{\leq n}=\left(  \prod_{i=1}^{n}\frac{\left(  1-q^{i}\right)
^{\binom{k+i-1}{k-1}}}{(1-q)^{k}}\right)  _{\leq n}.
\]
This implies that if $m\leq n$ then
\[
h_{k,n}(q)_{\leq m}=\left(  \frac{1}{(1-q)^{k(n-m)}}\prod_{i=1}^{m}%
\frac{\left(  1-q^{i}\right)  ^{\binom{k+i-1}{k-1}}}{(1-q)^{k}}\right)_{\leq m}  .
\]

We prove that if the conjecture is proved for $h_{k,m}(q)_{\leq m}$, then it is
true for $h_{k,n}(q)_{\leq m}$ for all $n\geq m$. \ This has been
checked using computer calculations for $k=2$ and $m\leq8$; and for $k=3$ the check has been
done for $m\leq6$. We then have verified the conjecture for all $k$ and $m\leq3$.
\\

The notion of stability of irreducible representations of $S_{n}$ has appeared
in several contexts. For one, the Kronecker coefficients have a stability
condition of Murnaghan \cite{Mur38}, \cite{Mur55}, \cite{KroneckerStability};
and there is stability through the character functions of Orellana and
Zabrocki \cite{OrellanaZabrocki}. More relevant to our work,
Church, Ellenberg, and Farb have shown, in their theory of FI-modules
\cite{FImodules}, that Schur functors satisfy a notion of $S_{n}$-stability
for sufficiently large $n$. Some of their stability assertions are made
precise by our identities. Section 4 contains two conjectures that would imply
a simple expression for the graded character of the $S_{n}$ coinvariants up to
degree $n$ (see also \cite{Bergeron}). This expression would also imply the
stability condition for the space of coinvariants, as indicated in \cite{FImodules}.

The key difference between our work and the earlier authors who observed stability relations of the type studied in this paper is our emphasis on the relationship between the representations of the symmetric group and those of  $GL_n$  (the so called Weyl modules), as studied in \cite{GW} through highest weight theory. Also, perhaps, novel to this paper is the use of the fact that Weyl modules of $GL_n$ restricted to $S_n$ (imbedded as permutation matrices) are direct sums of induced representations from appropriate subgroups.

\section{$S_{n}$ invariants in $S(\mathbb{C}^{k}\otimes\mathbb{C}^{n})$}

We associate to $S(\mathbb{C}^{k}\otimes\mathbb{C}^{n})$ the polynomial ring
$$
\ \mathcal{R} = \mathbb{C}[ x_{i,1},\dots, x_{i,n}: i=1,\dots, k ]
$$
in $k$ sets of variables, viewing $x_{i,j} $ as the generator $e_{i}
\otimes e_{j} \in\mathbb{C}^{k}\otimes\mathbb{C}^{n}.$ The diagonal action of
$s\in S_{n}$ on $\mathcal{R}$ is given by sending $x_{i,j}$ to $x_{i,s(j)}$ for all $i$. The polynomials of total degree $d$ will be denoted
by $\mathcal{R}_{d}$, and we let $\mathcal{R}_{\leq d}$ denote the space of
polynomials of total degree at most $d$.

A basis for the $S_{n}$ invariants of $\mathcal{R}$ can be found by first
choosing a monomial
$$
{\bf{x}}^{\alpha}= \prod_{i=1}^{k} \prod_{j=1}^{n} x_{i,j}^{  \alpha_{i,j}}%
$$
and taking the sum of the elements of the orbit $S_{n}({\bf{x}}^{\alpha})$. Note that the orbit
corresponds to taking the sequence of vectors $(P^{1}(\alpha) ,\dots,
P^{n}(\alpha))$, with
$$
P^{i}(\alpha) = ( \alpha_{1,i},\dots,  \alpha_{k,i}),
$$
and permuting the vectors in all possible ways, since $S_{n}$ acts diagonally.
This means that a representative of the orbit is the multiset of sequences
$P^{\alpha}= [P^{1}(\alpha) ,\dots, P^{n}(\alpha)].$ In other words, we look
at the set created by $P^{1}(\alpha) ,\dots, P^{n}(\alpha)$, including multiplicities.

The dimension of invariants of degree equal to $d$ is equal to the cardinality
of the set $\mathcal{P}_{d}^{n,k}$ of multisets $[P^{1},\dots,P^{n}]$ such
that each $P^{i}=(P_{1}^{i},\dots,P_{k}^{i})$ is a sequence of $k$ nonnegative
integers satisfying
$$
|P^{1}|+\cdots+|P^{n}|=d\text{ with }|P^{i}|=P_{1}^{i}+\cdots+P_{k}^{i}.
$$
The union of all $\mathcal{P}_{d}^{n,k}$ with $d\leq r$ will be denoted by
$\mathcal{P}_{\leq r}^{n,k}$. This implies, for the space of polynomials with
degree less than or equal to $d$, that
$$
\dim\mathcal{R}_{\leq d}^{S_{n}}=|\mathcal{P}_{\leq d}^{n,k}|.
$$

In \cite{Weyl}, Weyl showed that a set of generators for the invariants is
given by the polarized power sums
$$
p_{a}=p_{(a_{1},\dots,a_{k})}=\sum_{i=1}^{n} x_{1,i}^{a_{1}}\cdots x_{k,i}^{a_{k}}%
$$
with $0<|a|=a_{1}+\cdots+a_{k}\leq n$. In other words, every invariant can be
written as a polynomial in the polarized power sums of degree at most $n$.

The number of monomials $p_{a^{1}}\cdots p_{a^{r}}$ we can construct whose
total degree is less than or equal to $n$ is the same as picking a multiset of
sequences
$$
\left[  a^{1}=(a_{1}^{1},\dots,a_{k}^{1}),\dots,a^{r}=(a_{1}^{r},\dots
,a_{k}^{r})\right]  ,
$$
each with a nonzero entry, such that $|a^{1}|+\cdots+|a^{r}|\leq n$. Note
however that since each $a^{i}$ has a nonzero entry, there are at most $n$ of
them. This is in bijection with the set $\mathcal{P}_{\leq n}^{n,k}$ by adding
a sufficient number of sequences whose only entries are zeroes. This proves
that the set
$$
\{p_{a^{1}}\cdots p_{a^{r}}:|a^{1}|+\cdots+|a^{r}|\leq n\}
$$
is not only a spanning set for the space of invariants of degree at most $n$,
but also a basis. Surprisingly, there is a stronger independence satisfied by
these monomials. Though stated in terms of elementary symmetric functions, this independence follows from the work of Dalbec \cite{Dalbec} (or Vaccarino \cite{Vaccarino}). Here, we will use the polarized power sums and utilize the formalism of multisets.
\begin{theorem} \label{thm:independence}
The collection
$$
\{p_{a^{1}}\cdots p_{a^{r}}:|a^{1}|+\cdots+|a^{r}|\leq n+1\}
$$
is a basis for $\mathcal{R}_{\leq n+1}^{S_{n}}$.
\end{theorem}

\begin{proof}
We have shown that the set of monomials $p_{a^{1}}\cdots p_{a^{r}}$ of degree
at most $n$ are a basis for the invariants of degree at most $n$. For
monomials of degree $n+1$, we have to show that the number of possible
monomials in the polarized power sums of degree $n+1$ is equal to the
cardinality of $\mathcal{P}_{n+1}^{n,k}$.

An element $[P^{1},\dots, P^{n}] \in\mathcal{P}^{n,k}_{n+1}$ falls into one of
two cases:

\begin{enumerate}
\item $|P^{i}| \leq n$ for all $i$, or

\item there is an $i$ for which $|P^{i}| =n+1$.
\end{enumerate}

In the first case, we have the monomial
$$
p_{P^{1}}\cdots p_{P^{n}}.
$$
The convention here gives $p_{P^{i}} = 1$ for $P^{i} = (0,\dots, 0)$. This
accounts for all monomials in the polarized power sums with at most $n$
factors. In the second case, we may assume that $|P^{1}|=n+1$ and $|P^{i}|=0$
for $i>1$. The number of such choices is the number of ways of choosing
$P^{1}=(a_{1},\dots,a_{k})$ with
$$
a_{1}+\cdots+a_{k}=n+1,
$$
which is the binomial coefficient
$$
\binom{n+1+k-1}{k-1}=\binom{n+k}{k-1}.
$$
On the other hand, a monomial in the polarized power sums $p_{a^{1}}\cdots
p_{a^{n+1}}$ is of degree $n+1$ precisely when $|a^{i}|=1$ for all $i$. Let
$a_{j}$ be the number of $a^{i}$ whose component $j$ is nonzero. Then picking
such a monomial is equivalent to choosing $a_{1}+\cdots+a_{k}=n+1$, meaning
the number of such choices is again $\binom{n+k}{k-1}$. This shows that the
set of monomials in the polarized power sums whose degree is at most $n$ is a
basis for the invariants of degree at most $n+1$.
\end{proof}

We put the standard symmetric bilinear form $\left\langle (z_{1}%
,...,z_{k}),(w_{1}...,w_{k})\right\rangle =\sum z_{i}w_{i}$ on $\mathbb{C}%
^{k}$. On $\mathbb{C}^{k}\otimes\mathbb{C}^{n}$ we use the tensor product of
the forms on $\mathbb{C}^{k}$ and $\mathbb{C}^{n}$. This form is invariant
under the action of $S_{n}$ given by $\sigma(v\otimes w)=v\otimes\sigma w.$
Then $\mathcal{R}_{d}$ is isomorphic with $S^{d}(\mathbb{C}^{k}\otimes
\mathbb{C}^{n})$, using the form to identify $\mathbb{C}^{k}\otimes
\mathbb{C}^{n}$ with $(\mathbb{C}^{k}\otimes\mathbb{C}^{n})^{\ast}$. Thus we have

\begin{corollary}
\label{dimthm}For $r\leq m\leq n$,
$$
\dim S^{r}(\mathbb{C}^{k}\otimes\mathbb{C}^{m})^{S_{m}}=\dim S^{r}%
(\mathbb{C}^{k}\otimes\mathbb{C}^{n})^{S_{n}}%
$$
and
$$
\sum_{n=1}^{\infty}\dim S^{n}(\mathbb{C}^{k}\otimes\mathbb{C}^{n})^{S_{n}%
}q^{n}=%
{\displaystyle\prod_{r=1}^{\infty}}
\frac{1}{(1-q^{r})^{\binom{r+k-1}{k-1}}}
$$

\end{corollary}

\begin{proof}
In the above notation and with our identification, the space $\mathcal{R}_{r}^{S_{n}}=S^{r}%
(\mathbb{C}^{k}\otimes\mathbb{C}^{n})^{S_{n}}$ has dimension equal to
$|\mathcal{P}_{r}^{n,k}|$. It is therefore sufficient to show that
$|\mathcal{P}_{r}^{m,k}|=|\mathcal{P}_{r}^{n,k}|$. Since $n\geq m$, any
element $[P^{1},\dots,P^{n}]\in\mathcal{P}_{r}^{n,k}$ must have at most $r
\leq m$ nonzero $P^{i}$. Therefore, the map that takes $P\in\mathcal{P}%
_{r}^{m,k}$ and adds $n-m$ zero vectors is not only injective, but also
surjective. The two sets then have the same cardinality, and the equality holds.

We will now prove the sum product formula. We have shown that
$$
\dim S^{n}(\mathbb{C}^{k}\otimes\mathbb{C}^{n})^{S_{n}}=|\{p_{a^{1}}\cdots
p_{a^{r}}:|a^{1}|+\cdots+|a^{r}|=n\}|,
$$
which is the coefficient of $q^{n}$ in the generating series
$$
\prod_{a}\frac{1}{1-q^{|a|}},
$$
where the product runs over all nonzero $a\in\mathbb{Z}_{\geq 0}^{k}$. This product is
independent of $n$, and the factor $\frac{1}{1-q^{r}}$ appears with multiplicity the
number of $a$ with $|a|=r$. The number of $a=(a_{1},...,a_{k})$ satisfying
$\sum a_{i}=r$ is given by the binomial coefficient $\binom{r+k-1}{k-1}$.
\end{proof}
\\

We should mention that this generating function has interesting specializations when $k$ is specified. For instance, when $k= 2$, one gets the generating function enumerating plane partitions with a given trace. This is problem 7.99 in Stanley's book \cite{Stanley-Book-1999}. Other specializations give further interesting combinatorial connections.
\\

Recall that if $k,n\in\mathbb{Z}_{>0}$ then $\otimes^{k}\mathbb{C}^{n}$ is a
module for $S_{k}\times GL(n,\mathbb{C})$ with $s\in S_{k}$ permuting the
tensor factors and $g\in GL(n,\mathbb{C})$ acting by $\otimes^{k}g$. As a
representation of $S_{k}\times GL(n,\mathbb{C})$, the module $\otimes
^{k}\mathbb{C}^{n}$ decomposes according to Schur-Weyl duality (c.f.
\cite{GW},9.1.1) as follows:
$$
\otimes^{k}\mathbb{C}^{n}\cong\bigoplus_{%
\begin{array}
[c]{c}%
\lambda\vdash k\\
\ell(\lambda)\leq\min\{n,k\}
\end{array}
}V^{\lambda}\otimes{}_{n}F^{\lambda}.
$$
Here, $V^{\lambda}$ is the Young module for $S_{k}$ corresponding to the
partition $\lambda=(\lambda_{1}\geq\cdots\geq\lambda_{r}>0)$ of size
$|\lambda|=\lambda_{1}+\cdots+\lambda_{r}=k$ with length $\ell(\lambda)=r$;
and ${}_{n}F^{\lambda}$ is the Weyl module for $GL(n,\mathbb{C})$
corresponding to $\lambda$. Let $\varepsilon(\lambda)=\sum_{i=1}^{r}\lambda
_{i}\varepsilon_{i}$, though to simplify notation we may write $\varepsilon (\lambda) = \lambda$ since the $i^{\text{th}}$ coordinate of $\lambda$ is the coefficient of $\varepsilon_i$ in $\varepsilon(\lambda)$. Here, $\varepsilon_{i}$ is the linear functional on the
space of diagonal matrices
$$
h=\left[
\begin{array}
[c]{cccc}%
h_{1} & 0 & \cdots & 0\\
0 & h_{2} & \cdots & 0\\
\vdots & \vdots & \ddots & \vdots\\
0 & 0 & \cdots & h_{n}%
\end{array}
\right]
$$
given by $\varepsilon_{i}(h)=h_{i}$.
The space $\mathfrak{h}$ of diagonal
$n\times n$ matrices is a Cartan subalgebra of $M_{n}(\mathbb{C})$, the Lie
algebra of $GL(n,\mathbb{C})$; and ${}_{n}F^{\lambda}$ is the irreducible
representation of $GL(n,\mathbb{C})$ with highest weight $\varepsilon(\lambda)$
relative to the choice of positive roots $\varepsilon_{i}-\varepsilon_{j}$,
$1\leq i<j\leq n$. 

Consider the subgroup $GL(n-1,\mathbb{C})$ of $GL(n,\mathbb{C})$ consisting of
the matrices%
$$
\left\{  \left[
\begin{array}
[c]{cc}%
g & 0\\
0 & 1
\end{array}
\right]  \Big| ~g\in GL(n-1,\mathbb{C})\right\}  .
$$
Let $\mu=\sum_{i=1}^{n}\mu_{i}\varepsilon_{i}$ with $\mu_{i}\in\mathbb{Z}%
_{\geq0}$ and $\mu_{i}\geq\mu_{i+1}$ for all $i=1,\dots,n-1$ (so that $\mu$ is
a dominant integral weight). The branching theorem (c.f. \cite{GW}, Theorem
8.1.1) implies that if we consider ${}_{n}F^{\mu}$ to be a $GL(n-1,\mathbb{C}%
){}$ module, then
$$
{}_{n}F^{\mu}\Big|_{ GL(n-1,\mathbb{C})} \cong\bigoplus_{%
\begin{array}
[c]{c}%
\nu=\sum_{i=1}^{n-1}\nu_{i}\varepsilon_{i}\\
\mu_{1}\geq\nu_{1}\geq\mu_{2}\geq...\geq\mu_{n-1}\geq\nu_{n-1}\geq\mu_{n}%
\end{array}
}{}_{n-1}F^{\nu}.
$$
This implies that if $\ell(\mu)\leq m<n$ and if $GL(m,\mathbb{C})$ is embedded
in $GL(n,\mathbb{C})$ as%
$$
\left\{  \left[
\begin{array}
[c]{cc}%
g & 0\\
0 & I
\end{array}
\right]  \Big|~ g\in GL(m,\mathbb{C})\right\}  ,
$$
then ${}_{m}F^{\mu}$ occurs as a $GL(m,\mathbb{C})$ submodule of ${}_{n}%
F^{\mu}$.

Let $E_{i,j}\in M_{n}(\mathbb{C})$ be the matrix with all zero entries, except
for a $1$ in position $(i,j)$. One has an action of $M_{n}(\mathbb{C})$, as the
Lie algebra of $GL(n,\mathbb{C})$, on any polynomial representation of
$GL(n,\mathbb{C})$. The theorem of the highest weight implies that the weight
space $({}_{n}F^{\mu})_{\mu}$ is one dimensional, and the $\xi$ weight space of ${}_{n}F^{\mu}$ is spanned by the elements%
$$
E_{i_{1}+1,i_{1}}\cdots E_{i_{r}+1,i_{r}}v
$$
where $v$ is a nonzero element of $({}_{n}F^{\mu})_{\mu}$ and $\varepsilon(\mu
)-\sum_{j=1}^{r}(\varepsilon_{i_{j}}-\varepsilon_{i_{j}+1})=\xi$. Using this
observation we have

\begin{lemma}
\label{lemma2} If $\left\vert \mu\right\vert \leq m\leq n$ and $\xi$ is a
dominant weight of ${}_{n}F^{\mu}$, then $({}_{n}F^{\mu})_{\xi}=({}_{m}F^{\mu
})_{\xi}$.
\end{lemma}

\begin{proof}
Since $\dim({}_{n}F^{\mu})_{\mu}=1$ and $\ell(\mu) \leq m \leq n$, the
assertion is true for $\mu$. Since $\left\vert \mu\right\vert \leq m$,
$\varepsilon(\mu)=\sum_{i=1}^{m}\mu_{i}\varepsilon_{i}$; and since $\xi$ is
dominant, $\xi=\sum_{i=1}^{m}\xi_{i}\varepsilon_{i}.$ Thus if $\xi
=\varepsilon(\mu)-\sum_{j=1}^{r}(\varepsilon_{i_{j}}-\varepsilon_{i_{j+1}})$,
then the maximum of the $i_{j}$ occurring in the expression $E_{i_{1}+1,i_{1}%
}\cdots E_{i_{r}+1,i_{r}}v$ is less than or equal to $m-1$. Thus
$E_{i_{1}+1,i_{1}}\cdots E_{i_{r}+1,i_{r}}v\in({}_{m}F^{\mu})_{\xi}$.
\end{proof}

\begin{lemma}
\label{dimlemma} If $\left\vert \mu\right\vert \leq m\leq n$, then
$\dim\left(  {}_{m}F^{\mu}\right)  ^{S_{m}}\geq\dim\left(  {}_{n}F^{\mu
}\right)  ^{S_{n}}.$
\end{lemma}

\begin{proof}
If $\xi$ is a weight of $_{n}F^{\mu}$, then there exists an element $s$ of
$S_{n}$ such that $s\xi$ is dominant. Thus if $\left(  _{n}F^{\mu}\right)
_{\xi}$ is the corresponding weight space, there exists $s\in S_{n}$ such that
$s\xi$ is a weight of $_{m}F^{\mu}$ and $s\left(  _{n}F^{\mu}\right)  _{\xi
}=\left(  _{n}F^{\mu}\right)  _{s\xi}=\left(  _{m}F^{\mu}\right)  _{s\xi}$.
\ This implies that the span of $\left\{  s\left(  {}_{m}F^{\mu}\right)  |s\in
S_{n}\right\}  $ is $_{n}F^{\mu}$. We use the standard notation $\mathbb{C}%
S_{n}$ for the group algebra of $S_{n}$. We have a surjective $S_{n}$ module
homomorphism
$$
\mathrm{Ind}_{S_{m}}^{S_{n}}({}_{m}F^{\mu})=\mathbb{C}S_{n}\otimes
_{\mathbb{C}S_{m}}(_{{}m}F^{\mu})\rightarrow{}_{n}F^{\mu}\rightarrow0.
$$
Now Frobenius reciprocity implies that $\dim\left(  \mathbb{C}S_{n}%
\otimes_{\mathbb{C}S_{m}}(_{{}m}F^{\mu})\right)  ^{S_{n}}=\dim(_{{}m}F^{\mu
})^{S_{m}}$. This implies the inequality.
\end{proof}

The $GL(k,\mathbb{C})$-$GL(n,\mathbb{C})$ duality theorem says that as a
representation of $GL(k,\mathbb{C})\times GL(n,\mathbb{C})$,
$$
S^{r}(\mathbb{C}^{k}\otimes\mathbb{C}^{n})\cong%
{\displaystyle\bigoplus_{%
\begin{array}
[c]{c}%
\lambda\vdash r\\
\ell(\lambda) \leq\min\{n,k\}\mathrm{\ parts}%
\end{array}
}}
{}_{k}F^{\lambda}\otimes{}_{n}F^{\lambda}.
$$
In particular, this implies that if we consider the $S_{n}$ action on
$S^{r}(\mathbb{C}^{k}\otimes\mathbb{C}^{n})$ coming from the restriction of
the action from $GL(n,\mathbb{C})$, then%
$$
\dim S^{r}(\mathbb{C}^{k}\otimes\mathbb{C}^{n})^{S_{n}}=%
{\displaystyle\sum_{%
\begin{array}
[c]{c}%
\lambda\vdash r\\
\ell(\lambda) \leq\min\{n,k\}\mathrm{\ parts}%
\end{array}
}}
\dim{}_{k}F^{\lambda} \times \dim\left(  {}_{n}F^{\lambda}\right)  ^{S_{n}}.
$$

\begin{theorem}
\label{Weyl-stability}If $\lambda$ is a dominant integral weight of
\ $GL(m,\mathbb{C)}$ and $\left\vert \lambda\right\vert \leq m\leq n$, then
$$
\dim\left(  {}_{m}F^{\lambda}\right)  ^{S_{m}}=\dim\left(  {}_{n}F^{\lambda
}\right)  ^{S_{n}}.
$$

\end{theorem}

\begin{proof}
We have seen in Corollary \ref{dimthm} that if $r \leq m \leq n$ then
$$
\dim S^{r}(\mathbb{C}^{k}\otimes\mathbb{C}^{m})^{S_{m}}=\dim S^{r}%
(\mathbb{C}^{k}\otimes\mathbb{C}^{n})^{S_{n}}.
$$
This implies that%
$$
\sum_{%
\begin{array}
[c]{c}%
\lambda\vdash r\\
\ell(\lambda) \leq k
\end{array}
}\dim{}_{k}F^{\lambda}\times\left(  \dim\left(  {}_{m}F^{\lambda}\right)
^{S_{m}}-\dim\left(  {}_{n}F^{\lambda}\right)  ^{S_{n}}\right)  =0.
$$
If $k$ is the number of parts of $\lambda$ then $\dim{}_{k}F^{\lambda}>0$. The
previous lemma now implies the the differences are nonnegative, and thus every
term in the sum is nonnegative. For equality to hold, the differences must be
$0$, implying the result.
\end{proof}

In general if $G$ and $H$ are finite groups and $V$ is an $H$ module, then
$$
\mathrm{Ind}_{H}^{G}V=\mathbb{C}G\otimes_{\mathbb{C}H}V.
$$
We note that

\begin{lemma}
As an $S_{n}$ module%
$$
{}_{n}F^{\lambda}\cong\bigoplus_{\mu\text{ a dominant weight of }{}%
_{n}F^{\lambda}}\mathrm{Ind}_{S_{n,\mu}}^{S_{n}}({}_{n}F^{\lambda})_{\mu}%
$$
where $S_{n,\mu}=\{s\in S_{n}|s\mu=\mu\}$.
\end{lemma}

\begin{proof}
If $\mu$ is a weight of ${}_{n}F^{\lambda}$ then there exists $s\in S_{n}$
such that $s\mu$ is dominant. Thus since
$$
{}_{n}F^{\lambda}=\bigoplus_{\mu\text{ a weight of }{}_{n}F^{\lambda}}({}%
_{n}F^{\lambda})_{\mu}%
$$
we see that%
$$
{}_{n}F^{\lambda}=\sum_{s\in S_{n},\ \mu\text{ dominant}}s({}_{n}F^{\lambda
})_{\mu}.
$$
We note that if $\mu,\nu$ are dominant weights and if $s\in S_{n}$ is such
that $s\mu=\nu$, then $\mu=\nu$. Indeed, the theorem of the highest weight
implies that $s\mu=\mu-Q$ with $Q$ a non-negative integral combination of
elements of the form $\varepsilon_{i}- \varepsilon_{i+1}$ with $i=1,...,m$. Thus in particular,
$\left\langle \nu,Q\right\rangle \geq0$. Also, since $s\mu=\nu$, this implies
that $\mu=\nu+Q$. Then
$$
\left\langle \mu,\mu\right\rangle =\left\langle \nu,\nu\right\rangle
+2\left\langle \nu,Q\right\rangle +\left\langle Q,Q\right\rangle
\geq\left\langle \nu,\nu\right\rangle +\left\langle Q,Q\right\rangle .
$$
Since
$$
\left\langle \nu,\nu\right\rangle =\left\langle s\mu,s\mu\right\rangle
=\left\langle \mu,\mu\right\rangle ,
$$
this implies that $Q=0$. Thus $\nu=\mu$.

We therefore have%
$$
{}_{n}F^{\lambda}=\bigoplus_{\mu\text{ dominant}}\sum_{s\in S_{n}}s({}%
_{n}F^{\lambda})_{\mu}.
$$
We assert that the $S_{n}$ module $\sum_{s\in S_{n}}s({}_{n}F^{\lambda})_{\mu
}$ is equivalent with $\mathrm{Ind}_{S_{n,\mu}}^{S_{n}}({}_{n}F^{\lambda
})_{\mu}$. Indeed, if $s\in S_{n}$ and
$$
s({}_{n}F^{\lambda})_{\mu}=({}_{n}F^{\lambda})_{\mu},
$$
then $s\mu=\mu$ and thus $s\in S_{n,\mu}$. Let $s_{1},...,s_{r}$ be a set of
representatives for $S_{n}/S_{n,\mu}$. Then the elements $s_{1}\mu
,...,s_{r}\mu$ are distinct and%
$$
\sum_{s\in S_{n}}s({}_{n}F^{\lambda})_{\mu}=\bigoplus_{i=1}^{r}({}%
_{n}F^{\lambda})_{s_{i}\mu}=\bigoplus_{i=1}^{r}s_{i}({}_{n}F^{\lambda})_{\mu},
$$
so
$$
\dim\sum_{s\in S_{n}}s({}_{n}F^{\lambda})_{\mu}=r\dim({}_{n}F^{\lambda})_{\mu
}=\dim\mathrm{Ind}_{S_{n,\mu}}^{S_{n}}({}_{n}F^{\lambda})_{\mu}.
$$
Since the map%
$$
\mathbb{C}S_{n}\otimes_{\mathbb{C}S_{n,\mu}}({}_{n}F^{\lambda})_{\mu
}\rightarrow\sum_{s\in S_{n}}s({}_{n}F^{\lambda})_{\mu}%
$$
given by%
$$
s\otimes v\mapsto sv
$$
is surjective, we conclude that the map is an equivalence.
\end{proof}

Let $|\lambda| \leq m \leq n$. It will be important to see that the action of
$S_{n-\ell(\mu)} $ contained in $S_{n,\mu} = S_{m_{1}} \times\cdots\times
S_{m_{r}}\times S_{n-\ell(\mu)}$ is trivial on $({}_{n}F^{\lambda})_{\mu}=
({}_{\ell(\mu)} F^{\lambda})_{\mu}$. Let $l = \ell(\lambda)$. Embedding
$GL(l,\mathbb{C})$ in $GL(n,\mathbb{C})$ as before,%
$$
\left\{  \left[
\begin{array}
[c]{cc}%
g & 0\\
0 & I
\end{array}
\right]  \Big|~ g\in GL(l,\mathbb{C})\right\}  ,
$$
and embedding $GL(n-l,\mathbb{C})$ by
$$
\left\{  \left[
\begin{array}
[c]{cc}%
I & 0\\
0 & g
\end{array}
\right]  \Big|~ g\in GL(n-l,\mathbb{C})\right\}  ,
$$
gives an embedding of $GL(l,\mathbb{C}) \times GL(n-l,\mathbb{C})$ in
$GL(n,\mathbb{C}).$ The cyclic space of the $(\lambda_{1},\dots,\lambda
_{l},0,\dots,0)$ weight space under $GL(l,\mathbb{C})$ is equivalent to
${}_{l} F^{\lambda}$. As a representation of $GL(l,\mathbb{C}) \times
GL(n-l,\mathbb{C})$, the space ${}_{n}F^{\lambda}$ splits into a direct sum
equivalent to
$$
\bigoplus_{\xi_{1},\xi_{2}} m_{\lambda}(\xi_{1},\xi_{2}) ~{}_{l} F^{\xi_{1}}
\otimes{}_{n-l}F^{\xi_{2}}.
$$
for some multiplicities $m_{\lambda}(\xi_{1},\xi_{2}) $. Consider $m_{\lambda
}(\lambda,\xi)$. If it is nonzero, then $(\lambda,\xi)$ is a weight for
${}_{n}F^{\lambda}$, and so $|\lambda| +|\xi| = |\lambda|$. So $\xi= 0$. This
means that ${}_{n-l} F^{\xi}$ is one dimensional, and that the cyclic space
under $GL(l,\mathbb{C}) \times GL(n-l,\mathbb{C})$ of the $(\lambda_{1},\dots,
\lambda_{l},0,\dots, 0)$ weight space is equivalent to
$$
{}_{l} F^{\lambda}\otimes{}_{n-l}F^{0} = {}_{l}F^{\lambda} \otimes\mathbb{C},
$$
with $GL(n-l,\mathbb{C})$ acting trivially on $\mathbb{C}$.

\begin{lemma}
Let $\mu\neq0$ be a dominant weight of ${}_{n}F^{\lambda}$ and let $\ell(\mu)$
be the last index of $\mu$ which is positive. Then
$$
\ell(\mu) \geq\ell(\lambda).
$$

\end{lemma}

\begin{proof}
The theorem of highest weight implies the weights of ${}_{n}F^{\lambda}$,
other than\\ $(\lambda_{1},\dots, \lambda_{\ell(\lambda)},0,\dots,0)$ are of the form
$\mu = (\lambda_{1},\dots, \lambda_{\ell(\lambda)},0,\dots,0) - \xi$ with $(\xi
_{1},\dots,\xi_{n}) \neq0$ satisfying
$$
\xi_{1} \geq0, ~ \xi_{1}+\xi_{2} \geq0, \dots, ~ \xi_{1}+\cdots+\xi_{n-1}
\geq0, \text{ and } \xi_{1}+\cdots+\xi_{n} = 0.
$$
This implies that if $m$ is the last nonzero index of $\xi$, then $\xi_{m}<0$.
If $\ell(\mu) < \ell(\lambda)$, we must have $\xi_{\ell(\lambda)} =
\lambda_{\ell(\lambda)}$, with $\lambda_{\ell(\lambda)} >0$, since
$\lambda_{\ell(\lambda)}-\xi_{\ell(\lambda)} = \mu_{\ell(\lambda)} =0$. This
means we must have $m>\ell(\lambda)$, contradicting $\ell(\mu) < \ell
(\lambda)$.
\end{proof}

\begin{proposition}
\label{proptrivial} If $\mu$ is a dominant weight of ${}_{n}F^{\lambda}$, then
$({}_{\ell(\mu)} F^{\lambda})_{\mu}= ({}_{n}F^{\lambda})_{\mu}$. And if we
embed $GL(n-\ell(\mu),\mathbb{C})$ in $GL_{(}n,\mathbb{C})$ as
$$
\left\{  \left[
\begin{array}
[c]{cc}%
I & 0\\
0 & g
\end{array}
\right]  \Big|~ g\in GL(n-{\ell(\mu)},\mathbb{C})\right\}  ,
$$
then $GL(n-{\ell(\mu)},\mathbb{C})$ acts trivially on $({}_{\ell(\mu)}
F^{\lambda})_{\mu}$.
\end{proposition}

\begin{proof}
As is mentioned in the proof of Lemma \ref{lemma2}, we have that every element
of $({}_{n}F^{\lambda})_{\mu}$ is a linear combination of elements of the
form
$$
E_{j_{1},i_{1}}\cdots E_{j_{k},i_{k}}v
$$
with $v\in({}_{n}F^{\lambda})_{(\lambda_{1},\dots,\lambda_{\ell(\lambda)},0,\dots
,0)}$; and furthermore,
$$
(\lambda_{1},\dots,\lambda_{l},0,\dots,0)+\sum_{r=1}^{k}\varepsilon_{j_{r}%
}-\varepsilon_{i_{r}}=\mu.
$$
If $j_{u}$ is the maximum of the $j_{r}$, then $j_{u}\leq\ell(\mu)$ so that
all the $E_{j_{r},i_{r}}$ are in the Lie algebra of $GL(\ell(\mu),\mathbb{C}%
)$. This implies the first assertion. The second statement follows from the
above observations: The cyclic space for $GL(\ell(\mu),\mathbb{C})\times
GL(n-\ell(\mu),\mathbb{C})$ of $({}_{n}F^\lambda)_{(\lambda_{1},\dots,\lambda
_{\ell(\lambda)},0,\dots,0)}$ is ${}_{\ell(\mu)}F^{\lambda}%
\otimes\mathbb{C}$ with $GL(n-\ell(\mu),\mathbb{C})$ acting on $\mathbb{C}$ trivially.
\end{proof}

\section{Stability for general partitions}

Let $\mathcal{R}$ be, as before, the polynomial ring in $k$ sets of $n$
variables. For a given $L=(l_{1},\dots,l_{k})\in\mathbb{Z}^{k}_{\geq0}$, let
$V_{L} = S^{L}(\mathbb{C}^k \otimes \mathbb{C}^n)$ denote the space of homogeneous polynomials of degree $l_{i}$ in the
variables $ x_{i,1},\dots, x_{i,n}$. Clearly, $S^{L}(\mathbb{C}^{k}\otimes\mathbb{C}^{n})$ is invariant under the
action of $GL(n,\mathbb{C})$ and thus of $S_{n}$. Therefore,
$$
\mathcal{R}=\bigoplus_{L \in\mathbb{Z}^{k}_{\geq 0}}V_{L}%
$$
is a grading.
Each homogeneous component decomposes into irreducible $S_{n}$%
-representations
$$
V_{L}\cong\bigoplus_{\mu\vdash n} m_{L}^{\mu} V^{\mu}.
$$
with multiplicities given by
$$
m_{L}^{\mu} = \dim\mathrm{Hom}_{S_{n}}(V^{\mu},S^{L}(\mathbb{C}^{k}%
\otimes\mathbb{C}^{n})).
$$

Now Corollary \ref{dimthm} can be stated as follows: For $n\leq l$, we have
$$
\sum_{|L|=n}m_{L}^{(n)}=\sum_{|L| = n}m_{L}^{(l)}.%
$$

For any partition
$\mu$, the grading
$$
\mathrm{Hom}_{S_{n}}(V^{\mu},S(\mathbb{C}^{k}\otimes\mathbb{C}^{n}%
))=\bigoplus_{L}\mathrm{Hom}_{S_{n}}(V^{\mu},S^{L}(\mathbb{C}^{k}%
\otimes\mathbb{C}^{n}))
$$
has an associated Hilbert series%
$$
h_{k,n}^{\mu}(q_{1},...,q_{k})=\sum_{L}q_{1}^{l_{1}}q_{2}^{l_{2}}\cdots
q_{k}^{l_{k}} m^{\mu}_L.
$$

In this section, we are going to prove the following extension of Corollary \ref{dimthm}. Let $e_{1}$
denote the vector $(1,0,\dots)$. If $\mu$ is a partition of $n$ thought of as
an element of $\mathbb{Z}^{l(\mu)}$,%
$$
\mu+r e_{1}=(\mu_{1}+r,\mu_{2},\dots,\mu_{\ell(\mu)})
$$
is the partition obtained from $\mu$ by adding $r$ to the first component.

\begin{theorem} 
\label{multigradedthm} Let $\mu$ be a partition of $n$, and let $d \leq n-
\mu_{2}$. Then
$$
\sum_{l_{1}+\cdots+l_{k}= d } q_{1}^{l_{1}} \cdots q_{k}^{l_{k}} m^{\mu
}_{(l_{1},\dots,l_{k})} = \sum_{l_{1}+\cdots+l_{k}= d} q_{1}^{l_{1}} \cdots
q_{k}^{l_{k}} m^{\mu{+ r e_{1}} }_{(l_{1},\dots,l_{k})}.
$$

\end{theorem}

Before giving a proof of this theorem, we will first make a few remarks.
Note that Corollary \ref{dimthm} is a special case of this theorem by taking
$\mu=(n)$ so that $\mu_{2}=0$, choosing $d=n$, and setting all $q_{i}=1$.

One key aspect of the
proof of this stabilization result is the following description for the Hilbert series of $S(\mathbb{C}^{k}\otimes\mathbb{C}^{n})$. A proof of this description can be found
in \cite{Rom}; and the total degree version found in
Proposition 5.2 of \cite{OZHowe} can also be used to provide an alternate
proof of the stabilization. 
Fix a total order on the monomials in the $q_{1},...,q_{k}$ \ (e.g.
lexicographic order). A semi-standard filling of the Young diagram of $\mu$ is
a placement of monomials in the diagram such that they are in strictly
increasing order in the columns and weakly increasing order in the rows. The
weight of the filling is the product of the monomials in the filling. The
result used is that the coefficient of $q^{L}=q_{1}^{l_{1}}q_{2}^{l_{2}}\cdots
q_{k}^{l_{k}}$ in $h_{k,n}^{\mu}(q_{1},...,q_{k})$ is the number of fillings
of $\mu$ of weight $q^{L}$. This in plethystic notation says
$$
h_{k,n}^{\mu}(q_{1},...,q_{k})=s_{\mu}\left[  \frac{1}{(1-q_{1})\cdots
(1-q_{k})}\right]  .
$$
Contrast this with the corresponding Hilbert series
$$
\sum_{L}q^{L}\dim\mathrm{Hom}_{GL(n,\mathbb{C)}}({}_{n}F^{\lambda}%
,S^{L}(\mathbb{C}^{k}\otimes\mathbb{C}^{n}))=s_{\lambda}(q_{1},...,q_{k}).
$$
So one also has%
$$
s_{\mu}\left[  \frac{1}{(1-q_{1})\cdots(1-q_{k})}\right]  =\sum_{\lambda
}s_{\lambda}(q_{1},...,q_{k})\dim\mathrm{Hom}_{S_{n}}(V^{\mu},{}_{n}%
F^{\lambda}).
$$
\newline

\begin{proof}[Proof of Theorem \ref{multigradedthm}]
Every semistandard tableau of shape $\mu$ with monomial entries
in $q_{1},\dots,q_{k}$, and total degree $d\leq n-\mu_{2}$ has at least
$\mu_{2}$ cells filled by only $1$'s, meaning every cell in the first row and
below the second part must contain a $1$. Otherwise there would be more than
$n-\mu_{2}$ cells with at least one $q_{i}$, meaning the degree would be
larger than $n-\mu_{2}$. Therefore, the set of semistandard tableaux of shape
$\mu$ with degree less than or equal to $n-\mu_{2}$ is in bijection with the
set of semistandard tableaux of shape $\mu+re_{1}$ with degree less than or
equal to $n-\mu_{2}$. The bijection, adds $r$ cells with entry $1$ to the
first row. Here is an example:
$$
\begin{tikzpicture}[scale=.5] \coordinate (prev) at (0,0);
\foreach \dir in {7,3,2}{ \draw[help lines, line width = .25mm] (prev) -- +(0,1) coordinate (prev); \draw[help lines, line width = .25mm] (prev)+(0,-1) grid +(\dir,0); }; ;
\draw (.5,1.5) node {\tiny$q_1$};
\draw (.5,2.5) node {\tiny$q_2^2 $};
\draw (1.5,1.5) node {\tiny$q_1 $};
\draw (1.5,2.5) node {\tiny$q_2 q_3 $};
\draw (2.5,1.5) node {\tiny$q_4 $};
\draw (5.5,.5) node {\tiny$q_2 $};
\draw (6.5,.5) node {\tiny$q_5 $};
\draw (.5,.5) node {\tiny$1$};
\draw (1.5,.5) node {\tiny$1$};
\draw (2.5,.5) node {\tiny$1$};
\draw (3.5,.5) node {\tiny$1$};
\draw (4.5,.5) node {\tiny$1$};
\end{tikzpicture}~~\leftrightarrow
~~\begin{tikzpicture}[scale=.5] \coordinate (prev) at (0,0);
\foreach \dir in {11,3,2}{ \draw[help lines, line width = .25mm] (prev) -- +(0,1) coordinate (prev); \draw[help lines, line width = .25mm] (prev)+(0,-1) grid +(\dir,0); }; ;
\draw (.5,1.5) node {\tiny$q_1$};
\draw (.5,2.5) node {\tiny$q_2^2 $};
\draw (1.5,1.5) node {\tiny$q_1 $};
\draw (1.5,2.5) node {\tiny$q_2 q_3 $};
\draw (2.5,1.5) node {\tiny$q_4 $};
\draw (9.5,.5) node {\tiny$q_2 $};
\draw (10.5,.5) node {\tiny$q_5 $};
\draw (.5,.5) node {\tiny$1$};
\draw (1.5,.5) node {\tiny$1$};
\draw (2.5,.5) node {\tiny$1$};
\draw (3.5,.5) node {\tiny$1$};
\draw (4.5,.5) node {\tiny$1$};
\draw (5.5,.5) node {\tiny$1$};
\draw (6.5,.5) node {\tiny$1$};
\draw (7.5,.5) node {\tiny$1$};
\draw (8.5,.5) node {\tiny$1$};
\end{tikzpicture}
$$
\end{proof}

Viewing ${}_{n}F^{\lambda}$ as an $S_{n}$ module, one has
$$
{}_{n}F^{\lambda}\cong\bigoplus_{\mu\vdash n}\mathrm{Hom}_{S_{n}}(V^{\mu}%
,{}_{n}F^{\lambda})\otimes V^{\mu}.
$$
We will use the notation
$$
g_{\mu}^{\lambda}= \dim\mathrm{Hom}_{S_{n}}(V^{\mu},{}_{n}F^{\lambda}).
$$
The polynomial in Theorem \ref{multigradedthm} is symmetric in $q_1,\dots, q_k$, meaning the coefficient of
$s_\lambda(q_1,\dots, q_k)$ must be equal on both sides. This implies the following stability.

\begin{corollary}
For $\mu\vdash n$, $|\lambda| \leq n - \mu_{2}$ and $r \geq0$, we have
$$
g^{\lambda}_{\mu}= g^{\lambda}_{\mu+re_{1}}.
$$
\end{corollary}

\section{Stability conjectures and the Quasifree Conjecture}
The purpose of this section is to present a sequence of conjectures with interesting consequences. The first ``quasifree'' conjecture asserts that in the stable range, a classical theorem of Chevalley \cite{Chevalley} for one set of variables is true in $k$ sets of variables. This conjecture is equivalent to a statement about the truncated Hilbert series of the coinvariants that we show is a reformulation of certain conjectural remarks of Bergeron \cite{Bergeron, bosonfermion} for the special case of the Symmetric group and $k$ sets of commuting variables. 

The following conjectures assert a stability of the leading terms of the ideal generated by the invariants without constant term in $k$ sets of variables. The last part of the section gives a method for the proof of both conjectures that asserts that if the conjectures are true for k sets of variables and $n=m$, then they are true for $k$ sets of variables, $n \geq m$ and degrees less than or equal to $m$.\\

Let $\mathcal{I}^{k,n} = (\mathcal{R}^{k,n})^{S_n}$ be the invariants of the polynomial ring $\mathcal{R}^{k,n}$ in $k$ sets of $n$ variables. Let $\mathcal{I}^{k,n}_{\leq d}$ denote the subspace of invariant polynomials of degree $d$ or less, and let $\mathcal{I}^{k,n}_{+}$ denote the space of invariants with no constant term.

We have seen that 
\[
\mathcal{I}^{k,n} = \mathbb{C}[p_\alpha: \alpha \in \mathbb{Z}_{\geq 0 }^k, |\alpha| \leq n ].
\]
In particular, if we let $z_\alpha$ be a formal variable of degree $|\alpha|$, then Theorem \ref{thm:independence} states that the grade-preserving map
\[
T_{k,n}: \mathbb{C}[z_\alpha : |\alpha| \leq n] \rightarrow \mathcal{I}^{k,n} 
\]
given by $T_{k,n}(z_\alpha) = p_\alpha$ is bijective in degrees no greater than $n+1$. 
\\

To go between different polynomial rings, it will be convenient to add an extra subscript to $p_\alpha$ and set 
\[
p_{\alpha,n} = \sum_{j=1}^n x_{1,j}^{\alpha_1} \cdots x_{k,j}^{\alpha_k}
\]
For a given polynomial $f \in \mathcal{R}^{k,n}$, let $\res_m(f)$ be the polynomial one gets by specializing $x_{i,j}$ to $0$ for $j>m$. It is then clear that for $1 \leq m \leq n$, we have $\res_m(p_{\alpha,n} ) = p_{\alpha,m}$. Together with the observation about the map $T_{k,n}$ one has the following result.
\begin{lemma} \label{lemma:res}
If $1 \leq m \leq n+1$, then 
\[
\res_{m} : \mathcal{I}^{k,n}_{\leq m+1} \rightarrow \mathcal{I}^{k,m}_{\leq m+1}
\]
is a linear bijection.
\end{lemma}

Order elements of $(\mathbb{Z}_{\geq 0})^k$ lexicographically, so that $\alpha < \beta$ means that at the smallest index $i$ of disagreement, one has $\alpha_i < \beta_i$. The following conjecture has been checked by computer calculations for the following values of $(k,n)$: 

Let \[\pi: \mathcal{R}^{k,n} \rightarrow \mathcal{R}^{k,n} / \left(\mathcal{I}^{k,n}_+ \right)\]
denote the natural projection of the polynomial ring to its coinvariants. Since $\mathcal{I}^{k,n}$ is graded, the quotient is graded as well. Let $V^{k,n}$ be a graded subspace of $\mathcal{R}^{k,n}$ for which $\pi|_{V^{k,n}}$ gives a linear bijection. Then $V^{k,n}$ is a finite-dimensional subspace and gives a surjective map
\[
V^{k,n} \otimes \mathcal{I}^{k,n} \twoheadrightarrow \mathcal{R}^{k,n}.
\]
The following conjecture states that this map is also injective for degrees no larger than $n$:
\begin{conjecture}[Quasifreeness Conjecture]
Let $V^{k,n}$ be as above. For each $\alpha^1\leq \cdots \leq \alpha^r$, $\alpha^i \in (\mathbb{Z}_{\geq 0 })^k $ with $ d = \sum_i |\alpha^i| \leq n$, choose $v_{\alpha^1,\dots, \alpha^r} \in V^{k,n}_{\leq n-d}$. If
\[
\sum_{ \alpha^1\leq \cdots \leq \alpha^r } v_{\alpha^1,\dots, \alpha^r} p_{\alpha^1,n} \cdots p_{\alpha^r,n} = 0,
\]
then for all $\alpha^1,\dots,\alpha^r$, we have $v_{\alpha^1,\dots, \alpha^r} = 0$.

\end{conjecture}

We recall that if $V = \bigoplus_{ j \geq 0} V_j$ is a graded vector space over $\mathbb{C}$ with $\dim V_j < \infty$, then the Hilbert series of $V$ is the formal sum
\[
h_V(q) = \sum_{j \geq 0 } q^j \dim V_j.
\]
Let $h_{k,n}(q)$ denote the Hilbert series for $\mathcal{R}^{k,n}/\mathcal{I}^{k,n}_+$, meaning that we also have 
$ h_{k,n}(q) = h_{V^{k,n}}(q).
$

For the rest of this section we will be manipulating truncations of formal power series. If $u(q) = u_0 + u_1 q + u_2 q^2 + \cdots$ is a formal power series, we let $u(q)_{\leq m } = u_0 + \cdots + a_m q^m$ be the truncation of $u(q)$ up to degree $m$. The following facts may be easily proven, say by induction on $m$.
If $w(q) = w_0 + w_1 q + \cdots$ is another power series, then we will write $u(q) \leq w(q)$ if for each coefficient we have $u_i \leq w_i$. 
\\

Let $u(q), v(q),$ and $w(q)$ be monic formal power series (meaning $u_0= v_0 = w_0 = 1$). Then
\begin{enumerate}
\item If $(u(q)v(q))_{\leq m } = (w(q)v(q))_{\leq m }$, then $u(q)_{\leq m } = w(q)_{leq m}$. 
\item If $( u(q)/ w(q))_{\leq m}= 1$, then $u(q)_{\leq m} = w(q)_{\leq m}$. 
\item If $u(q) \leq w(q)$, then $u(q)_{\leq m} \leq w(q)_{\leq m}$.
\item Suppose $u(q),w(q) \geq 0$. If $u(q) \leq v^1(q)$ and $w(q) \leq v^2(q)$, then $u(q)w(q) \leq v^1(q)v^2(q).$
\end{enumerate}

\begin{lemma}
The Quasifree Conjecture is equivalent to
\[
h_{k,n}(q)_{\leq n} =  \left(  \prod_{i=1}^n \frac{ (1-q^i)^{ \binom{k+i-1}{k-1}}}{(1-q)^k}   \right)_{\leq n}
\]
\end{lemma}
\begin{proof}
First note that for $k,n \geq 1$,
\[
\sum_{i=1}^n \binom{k+i-1}{k-1} \geq kn.
\]
Therefore,
\[
\varphi(q) = \prod_{i =1 }^n \frac{ (1-q^i)^{\binom{k+i-1}{k-1} }  }{ (1-q)^k}
\]
is a monic polynomial. The Quasifree Conjecture is equivalent to saying that
\[
\left (   \frac{h_{k,n}(q)}{\prod_{i=1}^n (1-q^i)^{\binom{k+i-1}{k-1}}}        \right)_{\leq n }  = \left( \frac{1}{(1-q)^{kn}} \right)_{\leq n},
\]
which in turn can be rewritten as
\[
\left (   \frac{h_{k,n}(q)}{ (1-q)^{kn} \varphi(q) }        \right)_{\leq n }  = \left( \frac{1}{(1-q)^{kn}} \right)_{\leq n}
\]
The first property on formal power series states that we must then have 
\[
\left( \frac{h_{k,n}(q)}{\varphi(q)} \right)_{\leq n} = 1,
\]
and the second property gives us that $h_{k,n}(q)_{\leq n} = \varphi(q)_{\leq n}$, completing the lemma.
\end{proof}

This leads to the following conjecture.
\begin{conjecture} \label{conj:variant}
For $k\geq 1$ and $1 \leq m \leq n$, one has
\[
h_{k,n}(q)_{\leq m} = \left( \frac{1}{(1-q)^{k(n-m)}} h_{k,m}(q) \right)_{\leq m}.
\]
\end{conjecture}
We will prove that this conjecture follows from the Quasifree conjecture in Theorem \ref{thm:Qasifreeimpliesvariant}.
This variant conjecture allows for a wider range of computer calculations for varying values of $k,m,$ and $n$. We will now pursue a better understanding of these conjectures by a refined statement on the Gr\"{o}bner bases for these invariant ideals. 
\\

We start by ordering the variables of $\mathcal{R}^{k,n}$ so that $x_{i,j} > x_{p,q}$ if $q > j$ or $q=j$ and $p>i$. This means
\[
x_{1,1} > x_{2,1} > \cdots > x_{k,1} > x_{1,2} > \cdots > x_{k,2} > \cdots > x_{1,n} > \cdots > x_{k,n}.
\]

We then order the set of monomials in the $x_{i,j}$ by using the graded, reverse lexicographic ordering. For $g \in \mathcal{R}^{k,n}$, let $L(g)$ be its leading monomial in this ordering; and for a subset $S \subseteq \mathcal{R}^{n,k}$, let $L(S) = \{L(g) | g \in S\}$. Recall that a Gr\"{o}bner basis for an ideal $I$ is a set of generators $G$ for $I$ such that $L(G)$ generates the monomial ideal span  $L(I)$. Let $G_{k,n}$ be a Gr\"{o}bner basis for $(\mathcal{I}^{k,n}_+).$
\begin{proposition} \label{prop:Grobnerrestriction}
For $1 \leq m \leq n$, we have $\res_m(G_{k,n}) - \{0\}$ is a Gr\"{o}bner basis for $(\mathcal{I}^{k,m}_+).$
\end{proposition}
\begin{proof}
From the initial observations in Lemma \ref{lemma:res}, we have that $\res_m ( (\mathcal{I}^{k,n}_+ ) ) = ( \mathcal{I}^{k,m}_+ ).$ Suppose $f$ is a homogeneous polynomial for which $\res_m(f) \neq 0$. This means that it contains a monomial with no $x_{i,j}$ where $j > m$; and in our ordering, this monomial would lead any monomial in which $x_{i,j}$ with $j > m$ appears. This means that
\[
L(\res_m(f)) = L(f).
\]
The definition for a Gr\"{o}bner basis is that the span of $L((\mathcal{I}^{k,n}_+))$ is the ideal generated by $L(G_{k,n})$. We have
\begin{align*}
\Span   L ( ( \mathcal{I}^{k,m}_+ ) ) & = \Span L(\res_m(\mathcal{I}^{k,n}_+)) 
\\
&  = \Span \res_m ( L(\mathcal{I}^{k,n}_+)) = \res_m \Span L(\mathcal{I}^{k,n}_+).
\end{align*}
Since $\Span L(\mathcal{I}^{k,n}_+) = \mathcal{R}^{k,n} \Span L(G_{k,n})$, we have
\begin{align*}
\Span L(\mathcal{I}^{k,m}_+) = \res_m \Span L(\mathcal{I}^{k,n}_+) &= \res_m(\mathcal{R}^{k,n}) \res_m ( \Span L(G_{k,n})) \\
& = \mathcal{R}^{k,m} \Span L ( \res_m(G_{k,n})).
\end{align*}
This means $\res_m(G_{k,n}) - \{0\}$ is a basis for $(\mathcal{I}^{k,m}_+)$.
\end{proof}
Let $M_{k,m}$ be the collection of monomials in $\mathcal{R}^{k,m}$ not divisible by any element of $L(\mathcal{I}^{k,m}_+).$
\begin{corollary}
For $1 \leq m \leq n$, we have
\[
M_{k,n} \subseteq \mathbb{C}[x_{i,j} | j > m] M_{k,m}.
\]
\end{corollary}
\begin{proof}
For a composition $\alpha = (\alpha_1,\dots, \alpha_k)$, let
${\bf{x}}_i^\alpha = x_{1,i}^{\alpha_1} \cdots x_{k,i}^{\alpha_k}$. Then a monomial in $\mathcal{R}^{k,n}$ is an expression of the form $u  = \bf{x}_1^{\alpha^1} \cdots \bf{x}_k^{\alpha^k},$ where each $\alpha^i$ is a composition of length $k$. Suppose $u \in G_{k,n}$.

If $y \in \res_m L(\mathcal{I}^{k,n}_+)$ is a nonzero monomial, then it contains no $x_{i,j}$ with $j>m$ and also $u$ is not divisible by $y$. This means that $y$ does not divide $w = {\bf{x}}_1^{\alpha^1} \cdots {\bf{x}}_m^{\alpha^m}$. By the previous lemma, $y \in L(\mathcal{I}^{k,m}_+)$ and therefore $w \in G_{k,m}$. 

This means that
\[
u = w \prod_{j=m+1}^n {\bf{x}}_j^{\alpha^j} \in \mathbb{C}[x_{i,j} | j > m] M_{k,m} ,
\] 
establishing the claim.
\end{proof}

\begin{corollary}
For $1 \leq m \leq n$, 
\[
h_{k,n}(q) \leq \frac{h_{k,m}(q)}{(1-q)^{k(n-m)}}.
\]
\end{corollary}

\begin{theorem} \label{thm:Qasifreeimpliesvariant}
If the Quasifree Conjecture is true for $k$ and $m$, then Conjecture \ref{conj:variant} is true for $k$ and all $n \geq m$. 
\end{theorem}
\begin{proof}
We note that since $M_{k,n}$ gives a basis for $\mathcal{R}^{k,n}/(\mathcal{I}^{k,n}_+)$, we have
\[
\Span(M_{k,n} ) \mathbb{C}[ p_{\alpha,n} |~~ |\alpha|\leq n ] = \mathcal{R}^{k,n}.
\] 
Therefore, since $h_{k,n}(q) = h_{\Span(M_{k,n})}$ by the theory of Gr\"{o}bner bases and the previous corollary, 
\begin{align*}
\frac{1}{(1-q)^{kn}} & \leq \frac{ h_{k,n}(q)}{\prod_{i=1}^n (1-q^i)^{\binom{k+i-1}{k-1}}} \\
& \leq \frac{h_{k,m}(q)}{(1-q)^{k(n-m)}}   \frac{1}{\prod_{i=1}^n (1-q^i)^{\binom{k+i-1}{k-1}}}.
\end{align*}
Hence,
\begin{align*}
\left( \frac{1}{(1-q)^{kn}}  \right)_{\leq m} 
& \leq  \left(\frac{h_{k,m}(q)}{(1-q)^{k(n-m)}}   \frac{1}{\prod_{i=1}^n (1-q^i)^{\binom{k+i-1}{k-1}}} \right)_{\leq m} \\
& =    \left(\frac{h_{k,m}(q)}{(1-q)^{k(n-m)}}   \frac{1}{\prod_{i=1}^m (1-q^i)^{\binom{k+i-1}{k-1}}} \right)_{\leq m}  \\
& =    \left(\frac{h_{k,m}(q)_{\leq m}}{(1-q)^{k(n-m)}}   \frac{1}{\prod_{i=1}^m (1-q^i)^{\binom{k+i-1}{k-1}}} \right)_{\leq m} \\
& =    \left( \frac{ \frac{\prod_{j=1}^m (1-q^i)^{\binom{k+i-1}{m-1}}}{ (1-q)^{km}}    }{(1-q)^{k(n-m)}}   \frac{1}{\prod_{i=1}^m (1-q^i)^{\binom{k+i-1}{k-1}}} \right)_{\leq m}  \\
& = \left(\frac{1}{(1-q)^{kn}} \right)_{\leq m }.
\end{align*}
All the inequalities must therefore be equalities, and we get
\[
\left( \frac{ h_{k,n}(q)}{\prod_{i=1}^n (1-q^i)^{\binom{k+i-1}{k-1}} }    \right)_{\leq m} = \left( \frac{h_{k,m}(q)}{(1-q)^{k(n-m)} }    \frac{1 }{\prod_{i=1}^n (1-q^i)^{\binom{k+i-1}{k-1}}}        \right)_{\leq m}.
\]
The first property on power series truncations gives that this
implies
\[
h_{k,n}(q)_{\leq m} = \left( \frac{ h_{k,m}(q)}{(1-q)^{k(n-m)}} \right)_{\leq m}.
\]
\end{proof}

A Gr\"{o}bner basis $G$ is said to be reduced if for each $g \in G$, no leading term of $h \in G-\{g\}$ divides any monomial in the expansion of $g$.  We now assume that our Gr\"{o}bner basis is reduced. 
\begin{corollary}
Assuming the Quasifree Conjecture is true for $k$ and $m$, if $n \geq m$, then
\[
\res_m: (G_{k,n})_{\leq m } \rightarrow (G_{k,m} )_{\leq m}
\]
is bijective.
\end{corollary}
\begin{proof}
In light of Proposition \ref{prop:Grobnerrestriction}, it is clear that $\res_m((G_{k,n})_{\leq m}) - \{0\}$ is a set of elements of degree at most $m$ from a reduced Gr\"{o}bner basis for $(\mathbb{I}^{k,m}_+)$, meaning it equals $(G_{k,m})_{\leq m }$. If the map is not bijective, then there exists $g \in (G_{k,n})_{\leq m}$ for which $\res_m(g) = 0,$ and therefore every monomial in the expansion of $g$ must contain a variable $x_{i,j}$ with $j>m$. In particular, the leading monomial of $g$ contains an $x_{i,j}$ with $j>m$, and $L((G_{k,n})_{\leq m})$ contains a monomial not in $\mathbb{C}[x_{i,j} | j\leq m]$, meaning
\[
h_{k,n}(q)_{\leq m} \neq \left( \frac{ h_{k,m}(q)}{(1-q)^{k(n-m)}} \right)_{\leq m}.
\]
\end{proof}

This leads to the following conjecture on the Gr\"{o}bner basis:
\begin{conjecture}
For $1 \leq k$ and $1 \leq m \leq n$, one has $L((G_{k,n})_{\leq m}) = L( (G_{k,m})_{\leq m})$. 
\end{conjecture}

\bibliographystyle{abbrv}
\bibliography{stabilizations}

\hspace{1em} 

{\small  \noindent\newline Marino Romero\newline Universit\"{a}t Wien, 
\newline Fakult\"{a}t f\"{u}r Mathematik \newline\emph{E\--mail}: \texttt{marino.romero@univie.ac.at} 
}
\newline{\small  \noindent\newline Nolan
Wallach\newline University of California, San Diego\newline Department of
Mathematics  \newline\emph{E\--mail}: \texttt{nwallach@ucsd.edu} 
}

\end{document}